\documentclass[11pt]{amsart}

\usepackage[utf8]{inputenc}              
\usepackage{newtxtext,newtxmath}         
\usepackage[final]{microtype}            
\usepackage[dvipsnames]{xcolor}          
\usepackage[margin=3.2cm]{geometry}      

\expandafter\def\expandafter\normalsize\expandafter{%
  \normalsize
  \setlength\abovedisplayskip{7pt}%
  \setlength\belowdisplayskip{7pt}%
  \setlength\abovedisplayshortskip{5pt}%
  \setlength\belowdisplayshortskip{5pt}%
}

\usepackage{enumitem}
\setlist[itemize]{topsep=3pt,itemsep=2pt}
\setlist[enumerate]{topsep=3pt,itemsep=2pt}
\usepackage{setspace}
\usepackage{indentfirst}
\usepackage[hidelinks,bookmarksdepth=3]{hyperref}
\usepackage{hypcap}

\usepackage{csquotes}

\newcommand{\doi}[1]{\href{https://doi.org/#1}{doi:#1}}
\usepackage{amsmath,amssymb,amsfonts,amsthm}
\usepackage{mathtools}
\mathtoolsset{showonlyrefs=true}
\usepackage{mathrsfs}
\usepackage{bbm}
\usepackage{esint}
\usepackage{commath}
\usepackage{slashed}

\usepackage{graphicx}
\usepackage{tikz}
\usepackage{tikz-cd}
\usepackage{caption}
\graphicspath{{./Pictures/}}
\allowdisplaybreaks[1]

\usepackage{etoolbox}
\usepackage{combelow}
\usepackage{framed}

\numberwithin{equation}{section}

\theoremstyle{definition}

\newtheorem{theorem}{Theorem}
\newtheorem{lemma}[theorem]{Lemma}
\newtheorem{proposition}[theorem]{Proposition}
\newtheorem{corollary}[theorem]{Corollary}

\theoremstyle{definition}
\newtheorem{definition}[theorem]{Definition}

\author{Matilde Gianocca}
\address{Department of Mathematics, ETH Z\"urich, Switzerland}
\email{matilde.gianocca@math.ethz.ch}

\title{Rigidity in the Ginzburg--Landau approximation of harmonic spheres}


\hypersetup{
  pdftitle={Rigidity in the Ginzburg--Landau approximation of harmonic spheres},
  pdfauthor={Matilde Gianocca}
}

\begin{document}

\begin{abstract}
We prove that not every harmonic map from $S^{2}$ to $S^{2}$ can arise as a limit of Ginzburg--Landau critical points. More precisely, we show that the only degree-one harmonic maps that can be approximated in this way are rotations.

This conclusion follows from a rigidity theorem: we show that for every $\gamma>0$ and $\varepsilon$ small enough, the only critical points $u_\varepsilon:S^{2}\to\mathbb R^{3}$ of the Ginzburg--Landau energy $E_\varepsilon$ with energy below $8\pi-\gamma$ are (up to conjugation) rotations, that is $u_\varepsilon(x)=\sqrt{1-2\varepsilon^{2}}\;R_\varepsilon\,x$. The threshold is sharp: in \cite{KarpukhinSternJEMS} the authors constructed critical points with energy approaching $8\pi$ that are not of this form.
\end{abstract}

\maketitle

\section{Introduction}
The existence of harmonic maps between Riemannian manifolds has been studied for more than forty years.  
In their seminal work of the early 1980s, Sacks and Uhlenbeck~\cite{SacksUhlenbeck} proved that every non-aspherical manifold \(N\) admits a non-trivial \emph{harmonic sphere}, i.e.\ a harmonic map \(S^{2}\!\to\!N\).  
They introduced the perturbed \(\alpha\)-energy \(E_{\alpha}\) and analysed its critical points as \(\alpha\downarrow1\).

More recently, Lamm, Malchiodi and Micallef~\cite{LammMalchiodiMicallefJDG} showed that critical points of \(E_{\alpha}\) cannot approximate arbitrary harmonic maps \(S^{2}\!\to\!S^{2}\): in degree one, only \emph{rotations} occur, and low-energy \(E_{\alpha}\)-critical points must themselves be rotations when \(\alpha\) is sufficiently close to~\(1\).  
Hoerter, Lamm and Micallef~\cite{HoerterLammMicallef} later established an analogous rigidity for the so-called \(\varepsilon\)-harmonic maps.
In this paper we study the \emph{Ginzburg--Landau} (GL) energy
\[
E_{\varepsilon}(u)=\frac12\int_{S^{2}}|\nabla u|^{2}\,d\sigma
           +\frac1{4\varepsilon^{2}}\int_{S^{2}}(1-|u|^{2})^{2}\,d\sigma,
\qquad u\in W^{1,2}(S^{2},\mathbb R^{N}),
\]
whose critical points satisfy the Euler--Lagrange equation
\begin{equation}\label{eq:Gintroductionele}
-\Delta u=\frac1{\varepsilon^{2}}\,(1-|u|^{2})\,u.
\end{equation}
By the concentration--compactness results
Lin--Wang (\cite{LinWangI},\cite{LinWangII}), such critical points converge, in the bubble--tree sense, to harmonic maps.
Beyond its superconductivity origins, the GL relaxation \(E_{\varepsilon}\) has recently been used in the construction of maximal metrics in conformal eigenvalue
optimisation as well as in the construction of higher-dimensional harmonic maps by
Karpukhin--Stern \cite{KarpukhinSternInv,KarpukhinSternJEMS}.
Our main theorem shows that GL critical points cannot converge to arbitrary harmonic maps \(S^{2}\!\to\!S^{2}\) by classifying all low-energy
solutions for small \(\varepsilon\).
\begin{theorem}[Rigidity below \(8\pi\)]\label{theorem: maintheorem}
Given \(\gamma>0\) and a sequence \(u_{\varepsilon}:S^{2}\to\mathbb R^{3}\) of critical
points of \(E_{\varepsilon}\) with \(E_{\varepsilon}(u_{\varepsilon})\le 8\pi-\gamma\), there exists \(\varepsilon_{0}(\gamma)>0\) such that, for every
\(0<\varepsilon<\varepsilon_{0}\),
\[
u_{\varepsilon}(x)=\sqrt{1-2\varepsilon^{2}}\;R_\varepsilon\,x,\qquad R_\varepsilon\in SO(3).
\]
Consequently, \(u_{\varepsilon}\to R_\infty x\) smoothly as \(\varepsilon\downarrow0\) and $R_\infty\in SO(3)$.
\end{theorem}
Unlike the \(\alpha\)-energy and \(\varepsilon\)-harmonic approximations discussed
earlier, solutions of~\eqref{eq:Gintroductionele} do not take values in the prescribed
target \(S^{2}\).  A different strategy is therefore required to study rigidity.  
One of the main tools in our proof is a \textit{phase equation} for the unit field 
\(v_\varepsilon = u_\varepsilon/|u_\varepsilon|\), which, to our knowledge, has not been used in this setting before.

To prove Theorem \ref{theorem: maintheorem}, we first use the \textit{phase equation} for $\frac{u_\varepsilon}{|u_\varepsilon|}$ mentioned above, together with barycenter estimates for $u_\varepsilon$, to show that the distance of $u_\varepsilon$ from critical points given by rotations must be of order at most $\varepsilon^{4+\eta}$. Then, with a careful analysis of the second variation, we show that rotations have to be at distance at least $\varepsilon^4$ from other types of critical points, which allows us to conclude. It is important to note that the second variation at rotations has kernel directions which are orthogonal to the rotations themselves, and it is therefore not true that the second variation $D^2E_\varepsilon(cRx)$ is positive definite on the orthogonal complement of rotations. Due to this property, a careful understanding of the difference of two solutions is required in order to show that rotations must be isolated. These estimates are contained in Section \ref{section: thefirstmodeoftheddifference}.\\ \vspace{0.5cm}

\noindent \textbf{Acknowledgements:} I am thankful to my advisor T. Rivière, for his encouragement and support, as well as for his comments on this manuscript.
\section{Setting and Proof}

\subsection{Preliminaries. }
The following concentration-compactness properties satisfied by sequences of solutions to the GL equation have been shown by Lin and Wang:
\begin{theorem}[Bubble–tree convergence, \cite{LinWangII}]\label{thm:bubble-tree}
Let $N\ge 2$, $\Lambda>0$.  For each $\varepsilon>0$ let
\[
  u_\varepsilon \in C^{\infty}(S^{2},\mathbb R^{N})
\]
solve \eqref{eq:Gintroductionele} with $E_\varepsilon(e_\varepsilon)\leq\Lambda$.  Given a sequence $\varepsilon_{n}\downarrow0$,
there exist

\smallskip
\begin{itemize}[label=$\cdot$,leftmargin=1.5em]
  \item a subsequence (not relabelled),
  \item a harmonic map $u_{0}\colon S^{2}\to S^{N-1}$,
  \item non-constant harmonic spheres
        $\omega_{1},\dots,\omega_{k}\colon S^{2}\to S^{N-1}$,
  \item points $a_{1},\dots,a_{k}\in S^{2}$ and scales
        $\lambda_{1},\dots,\lambda_{k}\in(0,\infty)$,
\end{itemize}

\smallskip\noindent
such that:

\begin{enumerate}[label=\textup{(\roman*)},leftmargin=*,itemsep=6pt]
\item \(u_{\varepsilon_{n}}\to u_{0}\) in
      \(C^{\infty}\!\bigl(S^{2}\setminus\{a_{1},\dots,a_{k}\}\bigr)\).

\item \(\displaystyle
      \lim_{n\to\infty}E_{\varepsilon_{n}}(u_{\varepsilon_{n}})
      = E(u_{0}) + \sum_{i=1}^{k}E(\omega_{i})\).

\item For each \(i\),
      \(v_{n,i}(x):=u_{\varepsilon_{n}}(a_{i}+\lambda_{i}x)
        \to \omega_{i}(x)\) in \(C^{2}_{\mathrm{loc}}(\mathbb R^{2})\).

\item For every \(\delta>0\),
      \(\displaystyle
       \|u_{\varepsilon_{n}}-u_{0}-\sum \omega_i\|_{L^{\infty}(S^2)\cap H^1(S^2)}\to0.\)
\end{enumerate}
\end{theorem}

The concentration–compactness result in Theorem~\ref{thm:bubble-tree}
relies on the following $\varepsilon$–regularity lemma.

\begin{lemma}[$\varepsilon$–regularity, \cite{LinWangII}]\label{lem:eps-reg}
There exist universal constants
\(\varepsilon_{0}>0\), \(C_{0}>0\) and \(c>0\)
(depending only on \(N\)) such that the following holds.

\medskip
\noindent
Let \(u_{\varepsilon}:B_{2R}\to\mathbb{R}^{N}\) be a solution of
\textup{\eqref{eq:Gintroductionele}} with
\[
\int_{B_{2R}}\!
\Bigl(
      \tfrac12|\nabla u_{\varepsilon}|^{2}
      +\tfrac1{4\varepsilon^{2}}\bigl(1-|u_{\varepsilon}|^{2}\bigr)^{2}
\Bigr)\,dx
\;<\;
\varepsilon_{0},
\]
for some radius \(R\le1\).
Then the following pointwise bounds hold on the ball \(B_{R}\):

\begin{equation}\label{eq:eps-reg-first}
R^{2}\,
\sup_{x\in B_{R}}\Bigl(
      \tfrac12|\nabla u_{\varepsilon}|^{2}
      +\tfrac1{4\varepsilon^{2}}\bigl(1-|u_{\varepsilon}|^{2}\bigr)^{2}
\Bigr)
\;\le\;
C_{0},
\end{equation}

\begin{equation}\label{eq:eps-reg-second}
R^{2}\,
\sup_{x\in B_{R}}
\Bigl|
      \tfrac{1-|u_{\varepsilon}|^{2}}{\varepsilon^{2}}\,u_{\varepsilon}
\Bigr|
\;\le\;
C_{0}
\Bigl(
      \varepsilon^{2}+e^{-\,cR/\varepsilon}
\Bigr).
\end{equation}
\end{lemma}
\noindent
The combination of Theorem~\ref{thm:bubble-tree} with the
$\varepsilon$–regularity Lemma~\ref{lem:eps-reg} yields the following
uniform modulus estimate, which can be found in
\cite{DaLioGianocca}.

\begin{lemma}[Uniform modulus, \cite{DaLioGianocca}]\label{lem:unif-modulus}
Let $\{u_{\varepsilon}\}_{\varepsilon\downarrow0}$ be solutions of
\eqref{eq:Gintroductionele} satisfying the energy bound
\(E_{\varepsilon}(u_{\varepsilon})\le\Lambda\).
Then
\[
\bigl\|\,|u_{\varepsilon}|-1\bigr\|_{L^{\infty}(S^{2})}
\;\xrightarrow[\varepsilon\downarrow0]{}\;0,
\qquad\text{i.e.}\quad
|u_{\varepsilon}|\;\longrightarrow\;1 \text{ uniformly on } S^{2}.
\]
\end{lemma}

\subsection{Phase equation}

By Lemma~\ref{lem:unif-modulus}, the phase map  
\[
v_{\varepsilon}:=\frac{u_{\varepsilon}}{|u_{\varepsilon}|}
\quad\in\quad C^{\infty}\bigl(S^{2},S^{N-1}\bigr)
\]
is well--defined for every $\varepsilon$ sufficiently small.  
The next lemma shows that~$v_{\varepsilon}$ satisfies a nice PDE.

\begin{lemma}[Phase equation]\label{lem:phase-eq}
Let $u_{\varepsilon}\colon S^{2}\to\mathbb R^{N}$ be a critical point of
$E_{\varepsilon}$ with $\lvert u_{\varepsilon}\rvert>0$ on $S^{2}$, and set
$v_{\varepsilon}=u_{\varepsilon}/|u_{\varepsilon}|$. Then

\begin{equation}\label{eq:phase-laplacian}
-\Delta v_{\varepsilon}
=\;|\nabla v_{\varepsilon}|^{2}\,v_{\varepsilon}
\;+\;2\,\nabla v_{\varepsilon}\cdot\nabla\,\!\bigl(\ln|u_{\varepsilon}|\bigr),
\end{equation}
equivalently
\begin{equation}\label{eq:phase-divergence}
\operatorname{div}\,\!\bigl(|u_{\varepsilon}|^{2}\,\nabla v_{\varepsilon}\bigr)
  \;=\;
  -\,|u_{\varepsilon}|^{2}\,
    |\nabla v_{\varepsilon}|^{2}\,v_{\varepsilon}.
\end{equation}
In particular,
\begin{equation}\label{eq:phase-divT}
\bigl(\operatorname{div}(|u_{\varepsilon}|^{2}\nabla v_{\varepsilon})\bigr)^{T}=0
\end{equation}
and $v_\varepsilon$ is a weighted harmonic map from $S^2$ to $S^2$ with weight $|u_\varepsilon|^2$.
\end{lemma}

\begin{proof}\upshape
write \(v_\varepsilon=u_\varepsilon/|u_\varepsilon|\).

\medskip
{Laplacian of \(|u_\varepsilon|^{-1}\).}
\[
\begin{aligned}
-\Delta\,\!\Bigl(\frac{1}{|u_\varepsilon|}\Bigr)
&=-\operatorname{div}\,\!\Bigl(-\frac{\nabla|u_\varepsilon|}{|u_\varepsilon|^{2}}\Bigr) \\[2pt]
&=\operatorname{div}\,\!\Bigl(\frac{\nabla|u_\varepsilon|^{2}}{2|u_\varepsilon|^{3}}\Bigr) \\[2pt]
&=\frac{\Delta|u_\varepsilon|^{2}}{2|u_\varepsilon|^{3}}
  -\frac{3\,\nabla|u_\varepsilon|\!\cdot\!\nabla|u_\varepsilon|^{2}}{2|u_\varepsilon|^{4}} \\[2pt]
&=\frac{|\nabla u_\varepsilon|^{2}}{|u_\varepsilon|^{3}}
  -\frac{(1-|u_\varepsilon|^{2})}{\varepsilon^{2}|u_\varepsilon|}
  -\frac{3\,|\nabla|u_\varepsilon||^{2}}{|u_\varepsilon|^{3}} .
\end{aligned}
\]

\medskip
{Laplacian of \(v_\varepsilon\).}
\[
\begin{aligned}
-\Delta v_\varepsilon
&=-\Delta\,\!\Bigl(\frac{u_\varepsilon}{|u_\varepsilon|}\Bigr) \\[2pt]
&=-\frac{\Delta u_\varepsilon}{|u_\varepsilon|}
  -2\,\nabla u_\varepsilon\!\cdot\!\nabla\,\!\Bigl(\frac{1}{|u_\varepsilon|}\Bigr)
  -u_\varepsilon\,\Delta\,\!\Bigl(\frac{1}{|u_\varepsilon|}\Bigr) \\[4pt]
&=\frac{|\nabla u_\varepsilon|^{2}}{|u_\varepsilon|^{2}}\,v_\varepsilon
  -2\,\nabla u_\varepsilon\!\cdot\!\nabla\,\!\Bigl(\frac{1}{|u_\varepsilon|}\Bigr)
  -\frac{3\,|\nabla|u_\varepsilon||^{2}}{|u_\varepsilon|^{2}}\,v_\varepsilon ,
\end{aligned}
\]
\indent where the potential terms cancel because
\(-\Delta u_\varepsilon=\frac{1-|u_\varepsilon|^{2}}{\varepsilon^{2}}u_\varepsilon\).

\medskip
{Gradient identities.}

\begin{align}\label{equation: formulanablaveps}
\nabla v_\varepsilon
  &=\frac{\nabla u_\varepsilon}{|u_\varepsilon|}
     -\frac{u_\varepsilon}{|u_\varepsilon|^{2}}\nabla|u_\varepsilon| ,\\[4pt]
|\nabla v_\varepsilon|^{2}
  &=\frac{|\nabla u_\varepsilon|^{2}-|\nabla|u_\varepsilon||^{2}}{|u_\varepsilon|^{2}},\\[4pt]
\nabla v_\varepsilon\!\cdot\!\nabla\!\bigl(\,\ln|u_\varepsilon|\bigr)
  &=\frac{\nabla u_\varepsilon\!\cdot\!\nabla|u_\varepsilon|-v_\varepsilon|\nabla|u_\varepsilon||^{2}}
        {|u_\varepsilon|^{2}} .
\end{align}

\medskip
substituting the identities gives
\[
-\Delta v_\varepsilon
  =|\nabla v_\varepsilon|^{2}\,v_\varepsilon
  +2\,\nabla v_\varepsilon\!\cdot\!\nabla\!\bigl(\,\ln|u_\varepsilon|\bigr),
\]
\indent which is the stated phase equation.

\medskip
{Weighted divergence.}\par
Multiplying the phase equation by \(|u_\varepsilon|^{2}\) yields
\[
-|u_\varepsilon|^{2}\Delta v_\varepsilon
  =|u_\varepsilon|^{2}|\nabla v_\varepsilon|^{2}\,v_\varepsilon
   +2\,|u_\varepsilon|\,
     \nabla|u_\varepsilon|\cdot\nabla v_\varepsilon.
\]

\smallskip
Adding \(2|u_\varepsilon|\nabla|u_\varepsilon|\cdot\nabla v_\varepsilon\) to both sides:
\[
|u_\varepsilon|^{2}\Delta v_\varepsilon
 +2|u_\varepsilon|\,
   \nabla|u_\varepsilon|\cdot\nabla v_\varepsilon
 =-\,|u_\varepsilon|^{2}|\nabla v_\varepsilon|^{2}\,v_\varepsilon.
\]

\smallskip
The left-hand side is \(\operatorname{div}(|u_\varepsilon|^{2}\nabla v_\varepsilon)\), so
\[
\operatorname{div}(|u_\varepsilon|^{2}\nabla v_\varepsilon)
      =-\,|u_\varepsilon|^{2}\,|\nabla v_\varepsilon|^{2}\,v_\varepsilon.
\]
\indent Since the right–hand side is parallel to \(v_\varepsilon\), its tangential
component vanishes:
\[
\bigl(\operatorname{div}(|u_\varepsilon|^{2}\nabla v_\varepsilon)\bigr)^{T}=0 .
\]
\end{proof}

\subsection{Barycenter identities} We compute the barycenters of different Möbius transformations, which play a key role in the proof of Theorem \ref{theorem: maintheorem}. We also introduce for simplicity:
\begin{definition}
    A sequence $u_\varepsilon:S^2\to \mathbb{R}^3$ is said to be a $\gamma\,-$ admissible sequence if 
    \begin{itemize}[label=$\cdot$,leftmargin=1.5em]
        \item $-\Delta u_\varepsilon=\frac{1}{\varepsilon^2}(1-|u_\varepsilon|^2)u_\varepsilon$
        \item $E_\varepsilon(u_\varepsilon)<8\pi-\gamma$
        \item $u_\varepsilon$ is not constant.
    \end{itemize}
\end{definition}
\begin{lemma}\label{lemma: strong convergence of admissible seq}
    Let $\gamma>0$. Then, any $\gamma$-admissible sequence $\{u_\varepsilon\}_\varepsilon$ must converge strongly.
\end{lemma}
\begin{proof}
Consider the variation given by $u\circ \phi_t$, where $\phi_t$ is the flow of the conformal vector field $a-\langle a,x\rangle x$. Then, using the fact that the energy $E$ is conformally invariant,

\begin{align*}
    \frac{d}{dt}\bigg\vert_{t=0}E_\varepsilon (u_\varepsilon\circ\phi_t)=0& =\frac{d}{dt}\bigg\vert_{t=0}\int\frac{1}{2}|\nabla (u_\varepsilon\circ \phi_t)|^2+\frac{d}{dt}\bigg\vert_{t=0}\int\frac{1}{4\varepsilon^2}(1-|u_\varepsilon\circ \phi_t|^2)^2\\
    &=\frac{d}{dt}\bigg\vert_{t=0}\int\frac{1}{4\varepsilon^2}(1-|u_\varepsilon\circ \phi_t|^2)^2=-\frac{1}{\varepsilon^2}\int (1-|u_\varepsilon|^2)\langle u_\varepsilon,du_\varepsilon(\xi)\rangle\\
    &=\frac{1}{2\varepsilon^2}\int (1-|u_\varepsilon|^2)\langle d(1-|u_\varepsilon|^2),\xi\rangle=\frac{1}{4\varepsilon^2}\int \langle d(1-|u_\varepsilon|^2)^2,\xi\rangle \\
    & =\frac{1}{4\varepsilon^2}\int(1-|u_\varepsilon|^2)^2\mathrm{div}(\xi)= \frac{1}{4\varepsilon^2}\int(1-|u_\varepsilon|^2)^2 \langle x,a\rangle
\end{align*}
Multiplying everything by $\varepsilon^{-2}$, and using that $\frac{1}{\varepsilon^2}(1-|u|^2){\to}|\nabla u|^2$,
\begin{equation}\label{equation:balancingcondition}
    0=\int \frac{1}{\varepsilon^4}(1-|u_\varepsilon|^2)^2x_a.
\end{equation}
Assume by contradiction that the $\gamma$-admissible sequence develops bubbles. Then, assuming w.l.o.g. that the only bubble develops at $(1,0,0)$,
\begin{equation}
    |\nabla u_\varepsilon|^2\to 8\pi\delta_{(1,0,0)}
\end{equation}
and $|\nabla u_\varepsilon|^2,\frac{1-|u_\varepsilon|^2}{\varepsilon^2}\to 0$ in $C^\infty_{loc}(S^2\setminus\{(1,0,0)\})$.
\begin{align}\label{equation: positive part}
   & \int_{\{x_1\geq 0\}} \frac{(1-|u_\varepsilon|^2)^2}{\varepsilon^4}x_1\geq \int_{B_\delta(1,0,0)}  \frac{(1-|u_\varepsilon|^2)^2}{\varepsilon^4}x_1\geq C_\delta\bigg(\int_{B_\delta(1,0,0)}  \frac{(1-|u_\varepsilon|^2)}{\varepsilon^2}\bigg)^2\\&\geq C64\pi^2.
\end{align} 
On the other hand,
\begin{equation}\label{equation: negative part}
    \int_{\{x_1\leq 0\}}\frac{(1-|u_\varepsilon|^2)^2}{\varepsilon^4}x_1\to 0
\end{equation}
Combining \eqref{equation: positive part} and \eqref{equation: negative part},
\begin{equation}
    \int_{S^2} \frac{(1-|u_\varepsilon|^2)^2}{\varepsilon^4}x_1\geq C_\delta>0
\end{equation}
which is a contradiction and shows that the sequence cannot develop bubbles.
\end{proof}

Since the strategy to prove Theorem \ref{theorem: maintheorem} will be to show that $v_\varepsilon$ is close to harmonic maps from $S^2\to S^2$ with additional properties, we need to compute the quartic barycenters of all possible degree one harmonic maps of $S^2$. Recall that these maps are just the holomorphic and antiholomorphic degree on maps, that is Möbius transformations and their conjugates:
\begin{proposition}[Degree on harmonic maps fo $S^2$, \cite{LammMalchiodiMicallefJDG},\cite{EellsLemaireReport}] Let $w:S^2\to S^2$ be harmonic of degree one, then the map $\hat{w}:\mathbb{C}\to\mathbb{C}$ obtained via stereographic projection is given by
\[\hat{w}=R_1\circ D_{s^2}\circ \kappa^i\circ R_2\]
where $R_{1,2}\in SO(3)$ are rotations, $D_{s^2}=\mathrm{diag}(s,s^{-1})$ corresponds to the dilation $z\mapsto s^2z$ and $\kappa$ is complex conjugation. If $w$ is orientation preserving, $i=2$ and $\kappa^2=\mathrm{id}$ otherwise $i=1$.
\end{proposition}
We introduce the following notation:
\begin{definition}\label{definition: ulambda}
    $u_\lambda:= \Pi\circ D_\lambda\circ\Pi^{-1}:S^2\to S^2$, where $\Pi:\mathbb{C}\to S^2$ is the stereographic projection. Moreover, for an arbitrary harmonic map of degree one $w:S^2\to S^2$, we write $w:=w_\lambda$ where $\lambda$ is the parameter appearing in the decomposition $w=R_1\circ D\lambda\circ \kappa^i\circ R_2$.
\end{definition}
    
\begin{lemma}[quartic barycenter]\label{lemma: quarticbarycenter}
 For the stereographic dilations of Definition \ref{definition: ulambda},
    \begin{equation}
        \int |\nabla u_\lambda|^4x_3 = \frac{16\pi}{3\lambda^2}(1 - \lambda^4)
    \end{equation}
    \begin{equation}\label{equation: barycenterzeroinsymmcomponents}
        \int |\nabla u_\lambda|^4x_{1,2} = 0.
    \end{equation}
    In particular, for an arbitrary harmonic map $w_\lambda.S^2\to S^2$ (see Definition \ref{definition: ulambda}), 
    \begin{equation}\label{general barycenter}
        \bigg\vert\int_{S^2}|\nabla w_\lambda|^4x\, \bigg\vert = \bigg\vert \int |\nabla u_\lambda|^4x\,\bigg\vert = \bigg\vert\frac{16\pi}{3\lambda^2}(1 - \lambda^4)\,\bigg\vert
    \end{equation}
\end{lemma}
\begin{proof}

The formulas for the stereographic projections are given by
\[
\begin{aligned}
(X,Y) &= \left( \frac{x}{1 - z},\;\frac{y}{1 - z} \right),\\[4pt]
(x,y,z) &= \left(
  \frac{2X}{1 + X^{2} + Y^{2}},\;
  \frac{2Y}{1 + X^{2} + Y^{2}},\;
  \frac{-1 + X^{2} + Y^{2}}{1 + X^{2} + Y^{2}}
\right).
\end{aligned}
\]

The map $\tilde{u}_\lambda:= u_\lambda\circ\Pi:\mathbb{C}\to S^2$ is given by
\[
\tilde{u}_\lambda(X,Y)=\left(
  \frac{2\lambda X}{1 + r^2\lambda^2},\;
  \frac{2\lambda Y}{1 + r^2\lambda^2},\;
  \frac{-1 + r^2\lambda^2}{1 + r^2\lambda^2}
\right),\qquad r^{2}=X^{2}+Y^{2}.
\]

\begin{equation}\label{equation: nw1}
    |\nabla \tilde{u}_\lambda^1|^2=\bigg(-\frac{2\lambda X}{(1+r^2\lambda^2)^2}\lambda^22X+\frac{2\lambda}{1+r^2\lambda^2}\bigg)^2+\bigg(\frac{2\lambda X}{(1+r^2\lambda^2)^2}\lambda^22Y\bigg)^2
\end{equation}

\begin{equation}\label{equation: nw2}
    |\nabla \tilde{u}_\lambda^2|^2=\bigg(-\frac{2\lambda Y}{(1+r^2\lambda^2)^2}\lambda^22Y+\frac{2\lambda}{1+r^2\lambda^2}\bigg)^2+\bigg(\frac{2\lambda Y}{(1+r^2\lambda^2)^2}\lambda^22X\bigg)^2
\end{equation}

\begin{align}\label{equation: nw3}
    |\nabla \tilde{u}_\lambda^3|^2=&\bigg(\frac{2X\lambda^2}{1+r^2\lambda^2}-\frac{(r^2\lambda^2-1)(2X\lambda^2)}{(1+r^2\lambda^2)^2}\bigg)^2+\bigg(\frac{2Y\lambda^2}{1+r^2\lambda^2}-\frac{(r^2\lambda^2-1)(2Y\lambda^2)}{(1+r^2\lambda^2)^2}\bigg)^2\\
    =& \bigg(\frac{4X\lambda^2}{(1+r^2\lambda^2)^2}\bigg)^2 +\bigg(\frac{4Y\lambda^2}{(1+r^2\lambda^2)^2}\bigg)^2
\end{align}

\noindent Combining \eqref{equation: nw1}, \eqref{equation: nw2} and \eqref{equation: nw3},

\begin{align}
    |\nabla \tilde{u}_\lambda|^2=&\frac{1}{(1+r^2\lambda^2)^4}\bigg(16r^2\lambda^4+32\lambda^6X^2Y^2+16\lambda^6(Y^4+X^4)\bigg)\\&+\frac{1}{(1+r^2\lambda^2)^2}8\lambda^2-2\frac{1}{(1+r^2\lambda^2)^3}8\lambda^4(X^2+Y^2)\\=&\frac{16}{(1+r^2\lambda^2)^4}\bigg(r^2\lambda^4+\lambda^6r^4\bigg)+\frac{1}{(1+r^2\lambda^2)^2}8\lambda^2-2\frac{1}{(1+r^2\lambda^2)^3}8\lambda^4r^2\\
    =&\frac{8}{(1+r^2\lambda^2)^4}\bigg(2r^2\lambda^4+2\lambda^6r^4-2\lambda^4r^2-2r^4\lambda^6+\lambda^2(1+2r^2\lambda^2+r^4\lambda^4)\bigg)\\
    =&\frac{8\lambda^2}{(1+r^2\lambda^2)^4}(1+r^2\lambda^2)^2=\frac{8\lambda^2}{(1+r^2\lambda^2)^2}
\end{align}
Finally, 
\begin{equation}
    \bigl|\nabla_{S^2}u_\lambda\bigr|^4
=
\Bigl(2\lambda^2\,\frac{(1+r^2)^2}{(1+\lambda^2r^2)^2}\Bigr)^2
=
4\,\lambda^4\,\frac{(1+r^2)^4}{(1+\lambda^2r^2)^4},
\end{equation}
\begin{align}
    \int_{S^2}
x_{3}(z)\,\bigl|\nabla u_{\lambda}(z)\bigr|^{4}
&\;=\;
\int_{\mathbb{C}}
{\frac{r^{2}-1}{1+r^{2}}}
\;\Bigl[\,4\lambda^{4}\,\frac{(1+r^{2})^{4}}{(1+\lambda^{2}r^{2})^{4}}\Bigr]
\;{\frac{4}{(1+r^{2})^{2}}}\\&\;=\;32\pi\,\lambda^4
\int_{0}^{\infty}
\frac{r\,(r^2-1)\,(1+r^2)}{(1+\lambda^2r^2)^4}\,dr =  \frac{16\pi}{3\lambda^2}(1 - \lambda^4).
\end{align}
The second identity \eqref{equation: barycenterzeroinsymmcomponents} follows from the symmetry of $u_\lambda$ and \eqref{general barycenter} follows from the fact that $R,\kappa$ are isometries.
\end{proof}
\begin{corollary}\label{corollary: strong convergence to rotation}
    Let $\{u_\varepsilon\}_\varepsilon$ be a $\gamma$- admissible sequence. Then, $u_\varepsilon$ converges smoothly either to a rotation or its conjugate, $u_\infty(x)=(R\circ\kappa^i)x$, $R\in SO(3)$, $i\in {1,2}$.
\end{corollary}
\begin{proof}
By Lemma \ref{lemma: strong convergence of admissible seq},
    \begin{equation}
         0=\int \frac{1}{\varepsilon^4}(1-|u_\varepsilon|^2)^2x\to \int |\nabla u_\infty|^4x.
    \end{equation}
    By Lemma \ref{lemma: quarticbarycenter},
    \begin{equation}
        \int |\nabla u_\infty|^4x\iff u_\infty=R\circ \kappa^i.
    \end{equation}
    This concludes the proof of the Corollary.
\end{proof}
\begin{corollary}\label{corollary: apriorickepsestimate}
    Let $\gamma>0$. For any $\gamma$-admissible sequence $\{u_\varepsilon\}_\varepsilon$,
    \begin{equation}
        ||\nabla^k|u_\varepsilon|\,||_{L^\infty}\leq C_k\varepsilon^2.
    \end{equation}
\end{corollary}
\begin{proof}
    By Lemma \ref{lemma: strong convergence of admissible seq}, and Corollary \ref{corollary: strong convergence to rotation},
    \begin{equation}
        -\Delta u_\varepsilon u_\varepsilon =\frac{1-|u_\varepsilon|^2}{\varepsilon^2}|u_\varepsilon|^2\overset{C^\infty}{\to} -\Delta u_\infty u_\infty=|\nabla u_\infty|^2 = 2.
    \end{equation}
    so that
    \begin{equation}
        \frac{1-|u_\varepsilon|^2}{\varepsilon^2}\overset{C^\infty}{\to} 2.
    \end{equation}
    The claim follows by taking derivatives.
\end{proof}
\noindent By Lemma \ref{lemma: quarticbarycenter}, we see that there is an explicit relation between the quartic barycenter of a degree one harmonic map and its dilation component:
\begin{lemma}\label{lemma: dilation barycenter relation}
Assume a degree one harmonic map $w_\lambda:S^2\to S^2$ (using the notation of Definition \ref{definition: ulambda}) satisfies
\begin{equation}\label{equation: hypothesislemmabar-par}
\Bigl|\int_{S^2}x\,|\nabla w_\lambda|^4\Bigr|<\eta.
\end{equation}
Then
\[
|\lambda-1|\le \frac{3}{32\pi}\,\eta.
\]
\end{lemma}

\begin{proof}
By Lemma \ref{lemma: quarticbarycenter},
\[
\Bigl|\int_{S^2}x\,|\nabla w_\lambda|^4\Bigr|
=\bigg\vert\int_{S^2}x_3\,|\nabla u_\lambda|^4\bigg\vert
=\frac{16\pi}{3}\bigl\vert\lambda^{-2}-\lambda^2\bigr\vert.
\]
The hypothesis \eqref{equation: hypothesislemmabar-par} gives 
\begin{equation}
|\lambda^2-\lambda^{-2}|
=\bigg\vert\frac{(\lambda-1)(\lambda+1)(\lambda^2+1)}{\lambda^2}\bigg\vert\leq \frac{3}{16\pi}\eta.
\end{equation}
Since for all \(\lambda>0\),
\[
\frac{(\lambda+1)(\lambda^2+1)}{\lambda^2}
=\lambda + 1 + \lambda^{-1} + \lambda^{-2}
\;\ge\;2
\]
we have $\bigl|\,\lambda^2-\lambda^{-2}\bigr|
\;\ge\;2\,|\lambda-1|$ and hence
\[
2\,|\lambda-1|
\;<\;\frac{3}{16\pi}\,\eta
\quad\Longrightarrow\quad
|\lambda-1|<\frac{3}{32\pi}\,\eta,
\]
as claimed.
\end{proof}

\subsection{Closeness to the Möbius group} In this Section we recall the classical stability results for harmonic maps of the sphere in terms of energy defects. We then combine these with the phase equation satisfied by $v_\varepsilon$ to obtain quantitative estimates for the distance between $u_\varepsilon$ and rotations. We use the notation
\begin{equation}
    E(u):=\frac{1}{2}\int_{S^2}|\nabla u|^2,\quad \tau (u):=\Delta_{S^2}u+|\nabla u|^2 u.
\end{equation}

\begin{lemma}[{\cite[Lemma~1]{Topping97}}]\label{lem: energy gap tau}
There exist universal constants $\varepsilon_{0}>0$ and $\kappa>0$ such
that for any map $u\in C^{\infty}(S^{2},S^{2})$ of degree $\mathrm{deg}_{S^2}(u)=k$, if
\[
   E(u)\;-\;4\pi\,|k|\;<\;\varepsilon_{0},
\]
then
\[
   E(u)\;-\;4\pi\,|k|
   \;\le\;
   C\,
   \bigl\|\tau(u)\bigr\|_{L^{2}(S^{2})}^{2}.
\]
\end{lemma}
\begin{theorem}\label{thm:moebius-approx}(\cite{BMMS}, see also Theorem 1.1. \cite{ToppingRigidity}, \cite{HirschZemas})
There exists \(C>0\) such that for every map \(u\in W^{1,2}(S^{2},S^{2})\) of degree one
there exists a harmonic map \(v_\lambda:S^{2}\to S^{2}\) satisfying
\[
        \int_{S^{2}} |D(u-v_\lambda)|^{2}\,d\sigma
        \;\le\;
        C\,(\,E(u)-4\pi\,).
\]
\end{theorem}
We now prove that the maps $v_\varepsilon=\frac{u_\varepsilon}{|u_\varepsilon|}$, coming from a $\gamma$-admissible sequence have small tension field and small quartic barycenter, which allows us to prove that they must be close to a rotation by Lemma \ref{lem: energy gap tau} and Theorem \ref{thm:moebius-approx}.
\begin{lemma}\label{lemma: controlofdilationparameterintermsoftau}
    Let $\gamma>0$ and $\{u_\varepsilon\}_\varepsilon$ a $\gamma$-admissible sequence. Then, there exist a sequence $\lambda_\varepsilon$, $|1-\lambda_\varepsilon|\leq C\,(||\nabla|u_\varepsilon|\,||_2+|\nabla^2|u_\varepsilon|\,||_2)\leq  C \varepsilon^2$, and harmonic maps $w_{\lambda_\varepsilon}:S^2\to S^2$ satisfying
    \[
        \int |D(v_\varepsilon-w_{\lambda_\varepsilon})|^2\leq C\,(||\nabla|u_\varepsilon|\,||_2^2+|\nabla^2|u_\varepsilon|\,||_2^2)\leq  C\varepsilon^4,
    \]
    where $v_\varepsilon:=\frac{u_\varepsilon}{|u_\varepsilon|}$.
\end{lemma}
\begin{proof}
    By the phase equation of Lemma \ref{lem:phase-eq},
    \begin{equation}
        |\tau(v_\varepsilon)|^2=|2\nabla v_\varepsilon\cdot\nabla\ln\,|u_\varepsilon||^2.
    \end{equation}
    By Lemma \ref{lemma: strong convergence of admissible seq} and Corollary \ref{corollary: strong convergence to rotation},
    \begin{equation}
        \frac{1}{\varepsilon^2}|\nabla\ln|u_\varepsilon||=\bigg\vert\frac{1}{2|u_\varepsilon|^2}\nabla\frac{(1-|u_\varepsilon|^2)}{\varepsilon^2}\bigg\vert\to \frac{1}{2}|\nabla|\nabla u_\infty|^2|=0,
    \end{equation}
    which yields
    \begin{equation}\label{equation: eps4estimateoftau}
        ||\tau(v_\varepsilon)||_2^2\leq C\varepsilon^4 \int |\nabla v_\varepsilon|^2\leq C\varepsilon^4.
    \end{equation}
    By Theorem \ref{thm:moebius-approx} and Lemma \ref{lem: energy gap tau}, combined with \eqref{equation: eps4estimateoftau}, there exists a Möbius transformation $w_{\lambda_\varepsilon}$ such that
    \[
        \int |Dw_ {\lambda_\varepsilon}-Dv_\varepsilon|^2\leq C||\tau(v_\varepsilon)||_2^2\leq C\varepsilon^4.
    \]
    It remains to show that $|1-\lambda_\varepsilon|\leq C\varepsilon^2$. To do this we estimate the barycenter
    \begin{align}
        \bigg\vert\int (|\nabla w_{\lambda_\varepsilon}|^4-|\nabla v_\varepsilon|^4)\,x\,\bigg\vert&=\bigg\vert\int (|\nabla w_{\lambda_\varepsilon}|^2-|\nabla v_\varepsilon|^2)(|\nabla w_{\lambda_\varepsilon}|^2+|\nabla v_\varepsilon|^2)\,x\,\bigg\vert\\&=\bigg\vert\int (\nabla w_{\lambda_\varepsilon}-\nabla v_\varepsilon)(\nabla w_{\lambda_\varepsilon}+\nabla v_\varepsilon)(|\nabla w_{\lambda_\varepsilon}|^2+|\nabla v_\varepsilon|^2)\,x\,\bigg\vert\\&\leq \bigg(\int |D(w_{\lambda_\varepsilon}-v_\varepsilon)|^2\bigg)^{1/2}C(w_{\lambda_\varepsilon},v_\varepsilon)\\
        &\leq C||\tau(v_\varepsilon)||_2\leq C\varepsilon^2,
    \end{align}
    where in the last step we used that $v_\varepsilon$ is converging strongly to a rotation or its conjugate and hence there are uniform a priori bounds in arbitrary norm. Finally, by Lemma \ref{lemma: small veps barycenter for veps},
    \begin{equation}
        \bigg\vert \int |\nabla w_{\lambda_\varepsilon}|^4\,x\,\bigg\vert\leq  C\,(||\nabla|u_\varepsilon|\,||_2+|\nabla^2|u_\varepsilon|\,||_2)\leq C\varepsilon^2
    \end{equation}
    and by Lemma \ref{lemma: dilation barycenter relation}, $|\lambda_\varepsilon-1|\leq C\,(||\nabla|u_\varepsilon|\,||_2+|\nabla^2|u_\varepsilon|\,||_2)\leq C\varepsilon^2$.
\end{proof}

\noindent We can use the balancing condition \eqref{equation:balancingcondition} to obtain a quartic barycenter estimate for $v_\varepsilon$:
\begin{lemma}\label{lemma: small veps barycenter for veps}
    Let $u_\varepsilon$ be a $\gamma$-admissible. Then, there exists $\varepsilon_0>0$ such that for all $\varepsilon<\varepsilon_0$,
    \begin{equation}
      \bigg\vert  \int |\nabla v_\varepsilon|^4x \bigg\vert \leq C\,(||\nabla|u_\varepsilon|\,||_2+|\nabla^2|u_\varepsilon|\,||_2)\leq C\varepsilon^2
    \end{equation}
\end{lemma}
\begin{proof}
By strong convergence shown in Lemma \ref{lemma: strong convergence of admissible seq},
\begin{equation}
   \varepsilon^2 \bigg\vert\,\Delta\,\frac{1-|u_\varepsilon|^2}{\varepsilon^2}\bigg\vert=2\,\bigg\vert\frac{1-|u_\varepsilon|^2}{\varepsilon^2}|u_\varepsilon|^2-|\nabla u_\varepsilon|^2\bigg\vert\leq C\varepsilon^2
\end{equation}
so by \eqref{equation: formulanablaveps}
\begin{align}
   &\bigg\vert\,|\,\nabla v_\varepsilon|^4-\frac{(1-|u_\varepsilon|^2)^2}{\varepsilon^4}\bigg\vert=\bigg\vert\bigg(\frac{|\nabla u_\varepsilon|^2}{|u_\varepsilon|^2}-\frac{|\nabla|u_\varepsilon||^2}{|u_\varepsilon|^2}\bigg)^2 -\frac{(1-|u_\varepsilon|^2)^2}{\varepsilon^4}\bigg\vert\\
   \leq &\bigg\vert\big(\frac{1-|u_\varepsilon|^2}{\varepsilon^2}-\frac{|\nabla u_\varepsilon|^2}{|u_\varepsilon|^2}\big)\big(\frac{1-|u_\varepsilon|^2}{\varepsilon^2}+\frac{|\nabla u_\varepsilon|^2}{|u_\varepsilon|^2}\big)\bigg\vert +C\,(|\nabla|u_\varepsilon|\,|+|\nabla|u_\varepsilon|\,|^2)
   \\\leq & C\,(|\nabla|u_\varepsilon|\,|+|\nabla^2|u_\varepsilon|\,|)
\end{align}
Hence,
    \begin{equation}
        \int \frac{(1-|u_\varepsilon|^2)^2}{\varepsilon^4}x=0\implies \bigg\vert\int |\nabla v_\varepsilon|^4\,x\,\bigg\vert\leq C\,(||\nabla|u_\varepsilon|\,||_2+|\nabla^2|u_\varepsilon|\,||_2)\leq  C\varepsilon^2.
    \end{equation}
\end{proof}
\noindent As a corollary, we can prove that $v_\varepsilon$ must be quantitatively close to a rotation:
\begin{corollary}\label{corollary: epsfourclosenesstoidentity}
    In particular, up to a rotation and possible complex conjugation of $v_\varepsilon$,
    \begin{equation}
        \int |D(v_\varepsilon-\mathrm{id})|^2\leq C\,(||\nabla|u_\varepsilon|\,||_2^2+|\nabla^2|u_\varepsilon|\,||_2^2)\leq  C\varepsilon^4    
    \end{equation}
\end{corollary}
\begin{proof}
    To prove the Corollary it suffices to show that, for any harmonic map $w_\lambda:S^2\to S^2$, up to rotation and complex conjugation,
    \begin{equation}
        \int |D(w_\lambda-\mathrm{id})|^2\leq C|1-\lambda|^2.
    \end{equation}
    We have the explicit formula for the map $\tilde{w}_\lambda:= w_\lambda\circ\Pi:\mathbb{C}\to S^2$ given by
\[
\tilde{w}_\lambda(X,Y)=\left(
  \frac{2\lambda X}{1 + r^2\lambda^2},\;
  \frac{2\lambda Y}{1 + r^2\lambda^2},\;
  \frac{-1 + r^2\lambda^2}{1 + r^2\lambda^2}
\right),\qquad r^{2}=X^{2}+Y^{2}.
\]
Moreover,
\[
\tilde{\mathrm{id}}\,(X,Y)=\left(
  \frac{2 X}{1 + r^2},\;
  \frac{2 Y}{1 + r^2},\;
  \frac{-1 + r^2}{1 + r^2}
\right),\qquad r^{2}=X^{2}+Y^{2}.
\]
and hence their difference is given by
\begin{align}
(\tilde{w}_\lambda-\tilde{\mathrm{id}})(X,Y)&\bigg(\frac{2X(\lambda-1)(1-\lambda r^2)}{(1+\lambda^2 r^2)(1+r^2)},\frac{2Y(\lambda-1)(1-\lambda r^2)}{(1+\lambda^2 r^2)(1+r^2)},\frac{2r^2(\lambda+1)(\lambda-1)}{(1+\lambda^2r^2)(1+r^2)}\bigg)\\=&(\lambda-1)\bigg(\frac{2X(1-\lambda r^2)}{(1+\lambda^2 r^2)(1+r^2)},\frac{2Y(1-\lambda r^2)}{(1+\lambda^2 r^2)(1+r^2)},\frac{2r^2(\lambda+1)}{(1+\lambda^2r^2)(1+r^2)}\bigg)\\
=&(\lambda-1)G(X,Y)
\end{align}
By conformal invariance of the Dirichlet energy,
\begin{equation}
    \int_{S^2}|D(w_\lambda-\mathrm{id})|^2=(\lambda-1)^2\int_{\mathbb{C}}|\nabla G|^2\leq C(\lambda-1)^2.
\end{equation}
This concludes the proof of the Corollary.
\end{proof}
\subsection{Improved closeness to rotations} In this section we show that it is possible to improve the $\dot{H}^1$-estimate of Corollary \ref{corollary: epsfourclosenesstoidentity} to a stronger $W^{1,2}$ estimate for $u_\varepsilon-\mathrm{id}$.
\begin{proposition} \label{proposition: epsilon8closeness}Let $\{u_\varepsilon\}_\varepsilon$ be a $\gamma$-admissible sequence. Then, there exists $\eta> 0$ such that, up to rotations and possible complex conjugations of the $u_\varepsilon$,
    \begin{equation}
        \int_{S^2} |u_\varepsilon-\sqrt{1-2\varepsilon^2\,}\mathrm{id}|^2+\int_{S^2} |D(u_\varepsilon-\sqrt{1-2\varepsilon^2\,}\mathrm{id})|^2\leq C\varepsilon^{8+2\eta}.
    \end{equation}
\end{proposition}
\begin{proof}
    By Corollary \ref{corollary: epsfourclosenesstoidentity}, denoting by $h_\varepsilon:= v_\varepsilon-\mathrm{id}$,
    \begin{align}
       || -\Delta h_\varepsilon||_2^2&=\int -\Delta h_\varepsilon \big(|\nabla v_\varepsilon|^2v_\varepsilon-2\mathrm{id}+2\nabla v_\varepsilon\nabla \ln|u_\varepsilon|\big)\\&=\int -\Delta h_\varepsilon\big(2(v_\varepsilon-\mathrm{id})+v_\varepsilon(\nabla v_\varepsilon-\nabla \mathrm{id})(\nabla v_\varepsilon+\nabla \mathrm{id})+2\nabla v_\varepsilon\nabla \ln|u_\varepsilon|\big)\\
       &= \int -\Delta h_\varepsilon\big(2h_\varepsilon+v_\varepsilon\nabla h_\varepsilon(\nabla v_\varepsilon+\nabla \mathrm{id})+2\nabla v_\varepsilon\nabla \ln|u_\varepsilon|\big)\\
       &\leq \int 2|\nabla h_\varepsilon|^2+ |\nabla h_\varepsilon|^2|\nabla (v_\varepsilon(\nabla(v_\varepsilon+\mathrm{id})))|+2|\nabla h_\varepsilon||\nabla(\nabla v_\varepsilon\cdot\nabla\ln|u_\varepsilon|)|\\
       &\quad\quad+\frac{\sigma}{2}\int|\Delta h_\varepsilon|^2+\frac{1}{2\sigma}\int |\nabla h_\varepsilon|^2|\nabla(v_\varepsilon+\mathrm{id})|^2
    \end{align}
    which implies together with the a priori estimates for $|\nabla^k|u_\varepsilon||\leq C_k\varepsilon^2$ from Corollary \ref{corollary: apriorickepsestimate}
    \begin{equation}\label{equation:k1}
        ||\Delta h_\varepsilon||_2^2\leq C(||\nabla h_\varepsilon||_2^2+||\nabla |u_\varepsilon|||_2^2+||\nabla^2 |u_\varepsilon|||_2^2)\leq C\varepsilon^4
    \end{equation}
    By standard elliptic regularity estimates,
    \begin{equation}\label{equation: k2}
        ||\nabla^2h_\varepsilon||_{2}^2\leq C\varepsilon^4.
    \end{equation}
    Now we can use estimate \eqref{equation: k2} to get
    \begin{align}\label{equation: k3}
        \int |\nabla|\nabla v_\varepsilon|^2|^2&=\int |\nabla^2v_\varepsilon\cdot\nabla v_\varepsilon|^2\leq C\int |(\nabla^2v_\varepsilon-\nabla^2\mathrm{id})\nabla v_\varepsilon|^2+|\nabla^2\mathrm{id}\cdot\nabla v_\varepsilon|^2\\&= C\int |(\nabla^2v_\varepsilon-\nabla^2\mathrm{id})\cdot\nabla v_\varepsilon|^2+|\nabla^2\mathrm{id}\cdot(\nabla v_\varepsilon-\nabla\mathrm{id})|^2\\
        &\leq C||\nabla^2h||_2^2+C||\nabla h||_2^2\\
        &\leq C\varepsilon^4.
    \end{align}
Using $|\nabla v_\varepsilon|^2=\frac{|\nabla u_\varepsilon|^2-|\nabla|u_\varepsilon||^2}{|u_\varepsilon|^2}$ \eqref{equation: formulanablaveps},
\begin{equation}\label{equation: k4}
    \int \bigg\vert\,\nabla\bigg(\frac{|\nabla u_\varepsilon|^2}{|u_\varepsilon|^2}\bigg)\bigg\vert\,^2 \leq C\int |\nabla|\nabla v_\varepsilon|^2|^2+C\int |\nabla^2 |u_\varepsilon||^2+C\int |\nabla|u_\varepsilon||^2+C\int |\nabla|u_\varepsilon||^4\leq C\varepsilon^4
\end{equation}
Finally, we have the equation
\begin{equation}
   -\Delta|u_\varepsilon|^2\frac{1}{2|u_\varepsilon|^2}=\frac{1-|u_\varepsilon|^2}{\varepsilon^2}-\frac{|\nabla u_\varepsilon|^2}{|u_\varepsilon|^2}
\end{equation}
 and by elliptic regularity again
\begin{align}\label{equation: improvedboundsnablaabsu}
    \int \bigg\vert\,\nabla\bigg(\frac{1-|u_\varepsilon|^2}{\varepsilon^2}\bigg)\bigg\vert\,^2& \leq C\int |\nabla^3|u_\varepsilon||^2+C\int |\nabla^2|u_\varepsilon||^2+C\int |\nabla|u_\varepsilon||^2+C\int \frac{|\nabla|\nabla u_\varepsilon||^2}{|u_\varepsilon|^2}\\&\leq C\varepsilon^4
\end{align}
and hence 
\begin{align}
    ||\nabla |u_\varepsilon|\,||_2^2&\leq C\varepsilon^4\,\big(\,||\nabla^3|u_\varepsilon|\,||_2^2\, +\,||\nabla^2|u_\varepsilon|\,||_2^2\,+\,||\nabla|u_\varepsilon|\,||_2^2\,+\,||\nabla^2h_\varepsilon\,||_2^2\,+\,||\nabla h_\varepsilon\,||_2^2\,\big)\\
    &\overset{\eqref{equation:k1}}{\leq} C\varepsilon^4\,\big(\,||\nabla^3|u_\varepsilon|\,||_2^2\, +\,||\nabla^2|u_\varepsilon|\,||_2^2\,+\,||\nabla|u_\varepsilon|\,||_2^2\,\big)\\
    &\leq C_\theta\,\varepsilon^4\,\big(\,||\nabla|u_\varepsilon|\,||_2^{2(1-\theta)}||\nabla^{m_\theta}|u_\varepsilon|||_2^{2\theta}\, +\,||\nabla|u_\varepsilon|\,||_2^{2(1-\theta)}||\nabla^{n_\theta}|u_\varepsilon|||_2^{2\theta}\,+\,||\nabla|u_\varepsilon|\,||_2^2\,\big)\\
    &\leq C_\theta\,\varepsilon^{4+2\theta}||\nabla|u_\varepsilon|\,||_2^{2(1-\theta)},
\end{align}
where we used the a priori bounds coming from strong convergence (Lemma \ref{lemma: strong convergence of admissible seq}) $|\nabla^K|u_\varepsilon||\leq C_K\varepsilon^2$. Dividing both sides by $||\nabla|u_\varepsilon|\,||_2^{2(1-\theta)}$,
\begin{equation}
    ||\nabla|u_\varepsilon|\,||_2^{2}\leq C_\theta^{1/\theta}\,\varepsilon^{\frac{4+2\theta}{\theta}}
\end{equation}
Choosing $\theta$ small enough, and by interpolation again,
\begin{equation}\label{equation:estimateeps24}
   ||\nabla ^2|u_\varepsilon|\,||_2^2+||\tau(v_\varepsilon)||_2^2\leq C (||\nabla^2 |u_\varepsilon|\,||_2^2+||\nabla |u_\varepsilon|\,||_2^2)\leq C\varepsilon^{24}
\end{equation}
and by Corollary \ref{corollary: epsfourclosenesstoidentity},
\begin{equation}
    \int |\nabla(v_\varepsilon-\mathrm{id})|^2\leq C\varepsilon^{24}
\end{equation}
and we can use this to estimate
\begin{align}
    |\bar{u}_\varepsilon|:=\bigg\vert\int u_\varepsilon\bigg\vert &=\bigg\vert \frac{1}{\overline{1-|u_\varepsilon|^2}}\int (1-|u_\varepsilon|^2-\overline{1-|u_\varepsilon|^2})\,u_\varepsilon\,\bigg\vert\\&\leq  \frac{1}{\overline{1-|u_\varepsilon|^2}}\bigg(\int \Big(1-|u_\varepsilon|^2-\overline{1-|u_\varepsilon|^2}\Big)^2\bigg)^{1/2}\\
    &\leq   \frac{1}{\overline{1-|u_\varepsilon|^2}}\bigg(\int |\nabla (1-|u_\varepsilon|^2)|^2\bigg)^{1/2}\\
    &\leq \frac{C}{\varepsilon^2}\varepsilon^{12}=C\varepsilon^{10}.
\end{align}
With this barycenter estimate and \eqref{equation:estimateeps24}, we get
\begin{equation}
   \bigg| \int_{S^2}\frac{u_\varepsilon}{|u_\varepsilon|}-\overline{|u_\varepsilon|^{-1}}u_\varepsilon\,\bigg|\leq C\, \bigg\vert \int_{S^2}|\nabla|u_\varepsilon|^{-1}|^2\bigg\vert^{1/2}\leq C\varepsilon^{12}\implies \bigg\vert\int_{S^2}v_\varepsilon\,\bigg\vert\leq C\varepsilon^{24}
\end{equation}
\begin{equation}
    \int |v_\varepsilon-\mathrm{id}|^2\leq C\int |u_\varepsilon-\mathrm{id}-\bar{v}_\varepsilon|^2+|\bar{v}_\varepsilon|^2\leq C\varepsilon^{20}+\int |\nabla (v_\varepsilon-\mathrm{id})|^2\leq C\varepsilon^{20}.
\end{equation}
To conclude notice the following:
\begin{equation}
    v_\varepsilon-\mathrm{id}=\frac{1}{|u_\varepsilon|}(u_\varepsilon-|u_\varepsilon|\,\mathrm{id}) =\frac{1}{|u_\varepsilon|}(u_\varepsilon-\overline{|u_\varepsilon|}\,\mathrm{id}) +\frac{1}{|u_\varepsilon|}(\overline{|u_\varepsilon|}\,\mathrm{id}-|u_\varepsilon|\,\mathrm{id})
\end{equation}
which implies
\begin{equation}
    (u_\varepsilon-\overline{|u_\varepsilon|}\,\mathrm{id}) = |u_\varepsilon|(v_\varepsilon-\mathrm{id}) - (\overline{|u_\varepsilon|}-|u_\varepsilon|\,)\mathrm{id}
\end{equation}
and hence by Poincaré's inequality and the previous bounds,
\begin{equation}\label{w2}
    \int_{S^2}| u_\varepsilon-\overline{|u_\varepsilon|}\,\mathrm{id}\,|^2+ \int_{S^2}|\nabla( u_\varepsilon-\overline{|u_\varepsilon|}\,\mathrm{id})\,|^2\leq C\varepsilon^{20}.
\end{equation}
Finally it remains to understand
\begin{equation}
    \int\frac{1-|u_\varepsilon|^2}{\varepsilon^2}=\int -\Delta u_\varepsilon\,\frac{u_\varepsilon}{|u_\varepsilon|^2}=\int \frac{|\nabla u_\varepsilon|^2}{|u_\varepsilon|^2} +O(\varepsilon^{24})=\int |\nabla v_\varepsilon|^2+ O(\varepsilon^{24}).
\end{equation}
Moreover,
\begin{equation}
  \bigg\vert  \int |\nabla v_\varepsilon|^2-2\,\bigg\vert =\bigg\vert \int (\nabla v_\varepsilon-\nabla\mathrm{id})\,\cdot\, (\nabla v_\varepsilon+\nabla\mathrm{id})\,\bigg\vert \leq C\varepsilon^{12}
\end{equation}
so that
\begin{equation}\label{w1}
    \bigg\vert \int\frac{1-|u_\varepsilon|^2}{\varepsilon^2}-2\,\bigg\vert\leq C\varepsilon^{12}\implies \overline{|u_\varepsilon|^2} = 1-2\varepsilon^2+O(\varepsilon^{12})
\end{equation}
and finally,
\begin{equation}
    \bigg\vert\int |u_\varepsilon|^2-\overline{|u_\varepsilon|}\,^2\bigg\vert= \bigg\vert\int (|u_\varepsilon|-\overline{|u_\varepsilon|}\,)\,(|u_\varepsilon|+\overline{|u_\varepsilon|}\,)\bigg\vert\leq C\, \bigg\vert \int |\nabla|u_\varepsilon||^2\,\bigg\vert^{1/2}\leq C\varepsilon^{12}
\end{equation}
which combined with \eqref{w1} and \eqref{w2} implies
\begin{align}
       &\int_{S^2}| u_\varepsilon-\sqrt{1-2\varepsilon^2}\,\mathrm{id}\,|^2+ \int_{S^2}|\nabla( u_\varepsilon-\sqrt{1-2\varepsilon^2}\,\mathrm{id})\,|^2 \\=\,&\int_{S^2}| u_\varepsilon-\overline{|u_\varepsilon|}\,\mathrm{id}\,|^2+ \int_{S^2}|\nabla( u_\varepsilon-\overline{|u_\varepsilon|}\,\mathrm{id})\,|^2 + O(\varepsilon^{12})\\\leq&\,\, C\varepsilon^{12}.
\end{align}
This concludes the proof of the theorem with $\eta=2$.

\end{proof}

\subsection{Vectorial spherical harmonics}
Denote by $Y_{lm}$ the usual spherical harmonics satisfying
\begin{equation}
    \begin{cases}
        -\Delta Y_{lm} = l(l+1)Y_{lm}\\
        \int_{S^2}Y_{lm}Y_{l'm'}=\delta_l^{l'}\delta_m^{m'}. 
    \end{cases}
\end{equation}
Define the $\mathbb{R}^3$ valued functions 
\[
N_{lm}=Y_{lm}x,\quad \Psi_{lm}=\sqrt{l(l+1)}^{-1}\,\nabla Y_{lm},\quad \Phi_{lm}=\sqrt{l(l+1)}^{-1}\,x\times \nabla Y_{lm}.
\]
They satisfy
\begin{equation}\label{equation: propertiesofvsh}
    \begin{cases}
    \int_{S^2}N_{lm}N_{l'm'}=\int_{S^2}\Psi_{lm}\Psi_{l'm'}=\int_{S^2}\Phi_{lm}\Phi_{l'm'}=\delta_l^{l'}\delta_m^{m'}\\
        \int_{S^2}N_{lm}\Phi_{l'm'}=\int_{S^2}N_{lm}\Psi_{l'm'}=\int_{S^2}\Psi_{lm}\Phi_{l'm'}=0\\
        \forall l\geq 0,\,\int_{S^2}\Phi_{lm}=0, \forall l\neq 1 \int N_{lm}=\int \Psi_{lm}=0\\
        -\Delta N_{lm}=(2+l(l+1))N_{lm}-2 \nabla x\cdot\nabla Y_{lm}=(2+l(l+1))N_{lm}-2 \sqrt{l(l+1)}\Psi_{lm}\\
        -\Delta \Psi_{lm}= l(l+1)\Psi_{lm}-2 \sqrt{l(l+1)}N_{lm}\\
        -\Delta \Phi_{lm}= l(l+1)\Phi_{lm}
    \end{cases}
\end{equation}
Any function $N_1$ in the span of $\{N_{1m}\vert -1\leq m\leq 1\}$ can be written as $N_1(a):=\langle a\cdot x\rangle x$ and any function $\Psi_1$ in the span of $\{\Psi_{1m}\vert -1\leq m\leq 1\}$ can be written as $\Psi_1(b):=\nabla (b\cdot x)=b-\langle b,x\rangle x$ for vectors $a,b\in\mathbb{R}^3$. Using $\int_{S^2}x_ix_j=\frac{4\pi}{3}\delta_{ij}$ one shows
\begin{equation}\label{equation: vshpropertiesfirstmodes}
   \int N_1(a)=\frac{4\pi}{3}a\quad \int |N_1(a)|^2=\frac{4\pi}{3}|a|^2\quad \int \Psi_1(b)=\frac{8\pi}{3}b\quad \int |\Psi_1(b)|^2=\frac{8\pi}{3}|b|^2.
\end{equation}
\begin{equation}
    \int \nabla N_1(a)\cdot\nabla\Psi_1(b)=-\frac{16\pi}{3}a\cdot b
\end{equation}
\subsection{Second variation at rotations} Using vectorial spherical harmonics, we compute the second variation at the critical point $u_0:=\sqrt{1-2\varepsilon^2}\,x$. For a variation  $w$,
\begin{equation}
    D^2E_\varepsilon(u_0)[w,w]=\int_{S^2}|\nabla w|^2-2|w|^2+\frac{2}{\varepsilon^2}(u\cdot w)^2,
\end{equation}
and decomposing
\begin{equation}
    w = a_0x+N_1(a) + \Psi_1(b) +\sum\limits_{l\geq 2}\sum\limits_{|m|\leq l} a_{lm}N_{lm}+ b_{lm}\Psi_{lm}+c_{lm}\Phi_{lm},
\end{equation}
by the properties of vectorial spherical harmonics \eqref{equation: propertiesofvsh},
\begin{align}
    D^2E_\varepsilon(u_0)[w,w]&=\sum\limits_{l\geq 2}\sum\limits_{|m|\leq l}(l(l+1)+\frac{2(1-2\varepsilon^2)}{\varepsilon^2})a_{lm}^2+(l(l+1)-2)(b_{lm}^2+c_{lm}^2)\\&\,\,\,+4\sqrt{l(l+1)}a_{lm}b_{lm}+(2+\frac{2(1-2\varepsilon^2)}{\varepsilon^2})\frac{4\pi}{3}|a|^2-\frac{32\pi}{3}a\cdot b+(\frac{2(1-2\varepsilon^2)}{\varepsilon^2}-2)4\pi|a_0|^2
\end{align}
Notice that $l(l+1)-2-2\sqrt{l(l+1)}\geq \frac{1}{4} (l(l+1))$ for all $l\geq 3$ so that
\begin{align}\label{equation: secondvariationgeneralcoefficients}
    &D^2E_\varepsilon(u_0)[w,w]\\&\geq\sum\limits_{l\geq 3}\sum\limits_{|m|\leq l}(l(l+1)+\frac{2(1-2\varepsilon^2)}{\varepsilon^2}-2\sqrt{l(l+1)})a_{lm}^2+(l(l+1)-2-2\sqrt{l(l+1)})b_{lm}^2\\&+(l(l+1)-2)c_{lm}^2\,\,\,+(2+\frac{2(1-2\varepsilon^2)}{\varepsilon^2})\frac{4\pi}{3}|a|^2-\frac{32\pi}{3}a\cdot b+ |a_0|^2(l=2\,\,\text{modes})\\
    &\geq\frac{1}{4}\sum\limits_{l\geq 3}(l(l+1))(a_{lm}^2+b_{lm}^2+c_{lm}^2)+
    (6+\frac{2(1-2\varepsilon^2)}{\varepsilon^2}-4\sqrt{6})a_{2m}^2\\&+(6-2-\sqrt{6})b_{2m}^2+4c_{2m}^2
    \,\,\,+(2+\frac{2(1-2\varepsilon^2)}{\varepsilon^2})\frac{4\pi}{3}|a|^2-\frac{32\pi}{3}a\cdot b+|a_0|^2\\
    &\geq \frac{1}{4}\sum\limits_{l\geq 2}(l(l+1))(a_{lm}^2+b_{lm}^2+c_{lm}^2)+(2+\frac{2(1-2\varepsilon^2)}{\varepsilon^2})\frac{4\pi}{3}|a|^2-\frac{32\pi}{3}a\cdot b+|a_0|^2,
\end{align}
where the last step holds for all sufficiently small $\varepsilon>0$. Notice that for purely first tangential modes $w$, for instance $w=x$, the second variation vanishes and we certainly cannot show that the remaining terms above $(2+\frac{2(1-2\varepsilon^2)}{\varepsilon^2})\frac{4\pi}{3}|a|^2-\frac{32\pi}{3}a\cdot b$ are positive in general.
\subsection{The first mode of the difference of two solutions}\label{section: thefirstmodeoftheddifference} When $w=u-u_0$ is the difference of a critical point $u$ of $E_\varepsilon$ and $u_0=\sqrt{1-2\varepsilon^2}\,x$, the $l=1$ component of $w$ has some additional properties.

\begin{lemma}\label{lemma: usefulintegralsofw}
    For $w=u-u_0$ as above, $w=N_1(a)+\Psi_1(b)+w_\perp$, and $w_\perp$ has no $l=1$ component,
    \begin{equation}\label{k}
        \int |w|^2x\leq C|a|\,||w||_2+C||w_\perp||_2||w||_2
    \end{equation}
        \begin{equation}\label{kk}
        \int |w|^4x\leq C||w_\perp||_2||w||_{H^1}^3+C|a|\,||w||_{H_1}^3
    \end{equation}
\end{lemma}
\begin{proof}
   It suffices to show that 
   \begin{equation}\label{kkk}
       \int |\Psi_1(b)|^2\,x\,=\,\int |\Psi_1(b)|^4\,x=0.
   \end{equation}
   In order to do so, notice that $|\Psi_1(b)|^2= |b|^2-b_i^2x_i^2$ and hence \eqref{kkk} follows from $\int x_ix_j^2=\int x_i^2x_j^2x_k=0$. 
\end{proof}
The next Lemma shows that the coefficient $a$ is controlled by $w_\perp$:
\begin{lemma}\label{lemma:controlofa}
    In the setting $w=N_1(a)+\Psi_1(b)+w_\perp$, $ w=u-u_0$, $||w||_{H^1}\leq C\varepsilon^{4+\eta}$, and $w_\perp$ has no $l=1$ component,
    \begin{equation}
        |a|\leq \frac{C}{\varepsilon^2}||w_\perp||_2\,||w||_2
    \end{equation}
\end{lemma}
\begin{proof}
The balancing condition \eqref{equation:balancingcondition} gives
\begin{align}
   0= &\int_{S^2}\big(1-|u|^2\big)^2x =\int_{S^2}\big((1-|u_0|^2)-(2u_0\cdot\, w+|w|^2)\big)^2x\\=&\int_{S^2}\big(1-|u_0|^2\big)^2x-2\big(1-|u_0|^2\big)\cdot\,\big(2u_0\cdot\, w+|w|^2\big)x+
   \big(2u_0\cdot\, w+|w|^2\big)^2x\\=
   & \int_{S^2}-4(u_0\cdot\,w)\big(1-|u_0|^2\big)x-2|w|^2\big(1-|u_0|^2\big)x+|w|^4x+
   4|w|^2(u_0\cdot\, w)x+4(u_0\cdot\,w)^2x\\=&
   \int_{S^2}-8\varepsilon^2\sqrt{1-2\varepsilon^2}(w\cdot\,x)x-4\varepsilon^2|w|^2x+|w|^4x+4\sqrt{1-2\varepsilon^2}|w|^2(w\cdot\, x)x+4(1-2\varepsilon^2)(w\cdot\, x)^2x
\end{align}
For $w=N_1(a)+\Psi_1(b)+w_\perp$,
\begin{equation}\label{s}
     \int_{S^2}-8\,\varepsilon^2\,\sqrt{1-2\varepsilon^2}\,(w\,\cdot\,x)\,x\,=\,-\,\sqrt{1-2\varepsilon^2}\,\frac{32\pi\varepsilon^2}{3}\,a
\end{equation}
\begin{align}\label{ss}
     \Big|\,4\,(1-2\varepsilon^2)\,\int_{S^2}(w\cdot x)^2\,x\,\Big|\,= & \,C\,\big(\,|a|^2\,+\,||w_\perp||_2^2\,\big).
\end{align}
By Lemma \ref{lemma: usefulintegralsofw},
\begin{equation}\label{sss}
\Big|\,\int_{S^2}\,4\,\varepsilon^2\,|w|^2\,x\,\Big|\,\leq  C\varepsilon^2|a|^2+C\varepsilon^2||w_\perp||_2||w||_2
\end{equation}
and 
\begin{equation}\label{ssss}
\Big\vert\int_{S^2}\,|w|^4\,x\Big\vert\leq \,C\,|a|\,||w||_{H^1}^3\,+\,C\,||w||_{H^1}^3\,||w_\perp||_{H_1}
\end{equation}
Finally,
\begin{equation}
   \Big\vert \int \,4\,\sqrt{1-2\varepsilon^2}\,|w|^2\,(w\,\cdot\, x)\,x \Big\vert\leq \,C\,||w||_{H^1}^2\,(\,|a|\,+\,||w_{\perp}||_2\,).
\end{equation}
Combining \eqref{s}, \eqref{ss}, \eqref{sss} and \eqref{ssss}, and using $||w||_{H^1}\leq C\varepsilon^{2+\eta}$ to absorb higher order terms,
\begin{align}
    |a|\,\leq&\, \frac{C}{\varepsilon^2}\,||w_\perp||_2\,||w||_2
\end{align}
which concludes the proof of the Lemma.
\end{proof}
Next, we prove properties of the coefficient $b$:
\begin{lemma}
     In the setting $w=N_1(a)+\Psi_1(b)+w_\perp$, $ w=u-u_0$, $||w||_{H^1}\leq C\varepsilon^{4+\eta}$, and $w_\perp$ has no $l=1$ component,
   \begin{equation}
    |b|\leq {C}{\varepsilon^\eta}||w_\perp||_{H_1}
\end{equation}
\end{lemma}
\begin{proof}
Since $\int (1-|u|^2)u=0$,
\begin{equation}
    \int 2\varepsilon^2w=\int |w|^2 u +2\,(1-2\varepsilon^2) \,(\,w\cdot x\,)\,x+2\sqrt{1-2\varepsilon^2}\,(\,x\cdot w\,)\,w,
\end{equation}
which can be rearranged to obtain
\begin{align}
    &\frac{4\pi}{3}\big(\,2\varepsilon^2(a+2b)-2\,({1-2\varepsilon^2})\,a\,\big)=\int_{S^2} |w|^2 u+ 2\sqrt{1-2\varepsilon^2}\,(\,x\cdot w\,)\,w\\
    =&\int_{S^2} |w|^2 w+|w|^2 x+ 2\sqrt{1-2\varepsilon^2}\,(\,x\cdot w\,)\,w
\end{align}
and hence, using Lemma \ref{lemma: usefulintegralsofw},
\begin{align}
   |b|\leq |\frac{1-3\varepsilon^2}{2\varepsilon^2}a|+ \frac{C}{\varepsilon^2}(|a|+||w_\perp||_2)||w||_2+\frac{C}{\varepsilon^2}(|a|^3+|b|^3+||w_\perp||_{H^1}^3).
\end{align}
Using the bounds for $a$ obtained in Lemma \ref{lemma:controlofa} and absorbing the higher order therm,
\begin{equation}
    |b|\leq \frac{C}{\varepsilon^4}||w_\perp||_{H_1}||w||_{H_1}+ \frac{C}{\varepsilon^2}||w_\perp||_{H_1}||w||_{H_1}\leq  \frac{C}{\varepsilon^4}||w_\perp||_{H_1}||w||_{H_1}
\end{equation}
and by Proposition \ref{proposition: epsilon8closeness}, since $||w||_{H_1}\leq C\varepsilon^{4+\eta}$,
\begin{equation}
    |b|\leq {C}{\varepsilon^\eta}||w_\perp||_{H_1}
\end{equation}
\end{proof}
\begin{corollary}\label{corollary: wwperpcontrol}  In the setting $w=N_1(a)+\Psi_1(b)+w_\perp$, $ w=u-u_0$, $||w||_{H^1}\leq C\varepsilon^{4+\eta}$, and $w_\perp$ has no $l=1$ component,
    \begin{equation}
        ||w||\leq C||w_\perp||_{H^1}.
    \end{equation}
\end{corollary}
With these estimates, \eqref{equation: secondvariationgeneralcoefficients} becomes
\begin{align}\label{equation:positivesecondvariation}
   D^2E_\varepsilon(u_R)[w,w]\geq &\, \frac{1}{4}\sum\limits_{l\geq 2}(l(l+1))(a_{lm}^2+b_{lm}^2+c_{lm}^2)+\bigg(\,2+\frac{2(1-2\varepsilon^2)}{\varepsilon^2}\,\bigg)\frac{4\pi}{3}|a|^2-\frac{32\pi}{3}a\cdot b+|a_0|^2\\\geq& \, \frac{1}{8}||w_\perp||_{H^1}^2-C\varepsilon^{2\eta} ||\,w_\perp\,||_{H^1}^2\\\geq\,& \lambda_0\,||\,w\,||_{H^1}^2
\end{align}
where we used $\frac{1}{8}\big(\,l(l+1)+2\sqrt{l(l+1)}\,\big)+\frac{1}{8}\leq \frac{1}{4}l(l+1)$ for all $l\geq 2$ in the second line and Corollary \ref{corollary: wwperpcontrol} in the last line.
\subsection{$\varepsilon^{4+\eta}$-close critical points must be rotations}
By the positivity of second variation obtained in the previous subsection, we can now prove the following:
\begin{theorem}
    Let $\{u_\varepsilon\}_\varepsilon$ be a $\gamma$-admissible sequence. There exists $\varepsilon_0>0$ such that for all $\varepsilon<\varepsilon_0$, $u_\varepsilon$ is a rotation possibly composed with a complex conjugation:
    \begin{equation}
        u_\varepsilon(x)= R_\varepsilon \circ \kappa^i(x), \quad R_\varepsilon\in SO(3),\,i\in\{1,2\}.
    \end{equation}
\end{theorem}
\begin{proof}
    Firstly, replace $u_\varepsilon$ by $u_\varepsilon\circ \kappa$ if necessary, so that the closeness to the identity from Proposition \ref{proposition: epsilon8closeness} holds:
    \begin{equation}
        ||u_\varepsilon-\sqrt{1-2\varepsilon^2}\,\mathrm{id}||_{H^1}\leq C\varepsilon^{4+\eta}.
    \end{equation}
    We now write $u_\varepsilon=u_{R_\varepsilon}+w_\varepsilon$. In order to use the estimate \eqref{equation:positivesecondvariation}, we need to show that $w_\varepsilon$ can be chosen to have no $\Phi_{1m}$ component. This can be done as follows:  
    Minimize the distance 
    \begin{equation}
       F(R_0)=\min\limits_{R\in SO(3)} F(R)=\min\limits_{R\in SO(3)}\bigg\{\int_{S^2}|u_\varepsilon(x)-\sqrt{1-2\varepsilon^2}\,R(x)|^2\bigg\}.
    \end{equation}
    Then $w_\varepsilon=u_\varepsilon\circ R_0-\sqrt{1-2\varepsilon}\,\mathrm{id}$ (we rotate the map without re-labeling) has no $\Phi_{1m}$ components: indeed,
    \begin{align}
        0=&\frac{d}{dt}\bigg\vert_{t=0}\int |u_\varepsilon\circ R_0-\sqrt{1-2\varepsilon^2}\,R_tR_0(x)|^2=\frac{d}{dt}\bigg\vert_{t=0}\int -2\sqrt{1-2\varepsilon^2\,}(u_\varepsilon\circ R_0(x))R_tR_0(x)\\& = -2\sqrt{1-2\varepsilon^2}\int (u_\varepsilon\circ R_0)(x)\frac{d}{dt}\bigg\vert_{t=0}R_tR_0(x)=-2\sqrt{1-2\varepsilon^2}\int (u_\varepsilon\circ R_0)(x)\cdot a({R_t})\times R_0(x)
    \end{align}
    and since $a(R_t)$ is arbitrary, the claim follows. Up to a rotation we have therefore obtained that $w_\varepsilon=u_\varepsilon-\sqrt{1-2\varepsilon^2}\,\mathrm{id}$ decomposes in vectorial spherical harmonics $w_\varepsilon=\Psi_1(b_\varepsilon)+N_1(a_\varepsilon)+w_{\perp,\varepsilon}$ and $w_\perp$ has no $l=1$ component. We can now expand the expression $X\mapsto DE_\varepsilon(u_0+X)[w_\varepsilon]\in\mathbb{R}$ into
    \begin{align}
        DE_\varepsilon(u_\varepsilon)[w_\varepsilon]&=DE_\varepsilon(\sqrt{1-2\varepsilon^2}x+w_\varepsilon)[w_\varepsilon]\\&=DE_\varepsilon(\sqrt{1-2\varepsilon^2}x)[w_\varepsilon]+D^2E_\varepsilon(\sqrt{1-2\varepsilon^2}x)[w_\varepsilon,w_\varepsilon]\\&\quad\quad+D^3E_\varepsilon(\sqrt{1-2\varepsilon^2}x)[w_\varepsilon,w_\varepsilon,w_\varepsilon]+D^4E_\varepsilon(\sqrt{1-2\varepsilon^2}x)[w_\varepsilon,w_\varepsilon,w_\varepsilon,w_\varepsilon],
    \end{align}
    which is
    \begin{align}
        &DE_\varepsilon(\sqrt{1-2\varepsilon^2} x+w_\varepsilon )[w_\varepsilon ]\\&
=\int_{S^2}\,|\nabla w_\varepsilon |^{2}-2\,|w_\varepsilon |^{2}
+\frac{1-2\varepsilon^{2}}{\varepsilon^{2}}(x\!\cdot\! w_\varepsilon )^{2}
+\frac{3\sqrt{1-2\varepsilon^{2}}}{\varepsilon^{2}}(x\!\cdot\! w_\varepsilon )\,|w_\varepsilon |^{2}
+\frac{1}{\varepsilon^{2}}|w_\varepsilon |^{4}
\end{align}
and by \eqref{equation:positivesecondvariation}, for all $\varepsilon$ small enough,
\begin{align}
&\geq \lambda_0 ||\,w_\varepsilon \,||_{H^1}^2-\frac{3\sqrt{1-2\varepsilon^{2}}}{\varepsilon^{2}}||\,w_\varepsilon \,||_{H^1}^3
-\frac{1}{\varepsilon^{2}}||\,w_\varepsilon \,||_{H^1}^4\\
&\geq \frac{\lambda_0}{2}||\,w_\varepsilon \,||_{H^1}^2
    \end{align}
    which contradicts the fact that $DE_\varepsilon(u_\varepsilon)\equiv 0$. Hence $w_\varepsilon=0$, and this concludes the proof of the Theorem.   
\end{proof}

\end{document}